\documentclass[12pt,reqno]{article}

\usepackage[usenames]{color}
\usepackage{amssymb}
\usepackage{graphicx}
\usepackage{amscd}
\usepackage{booktabs}
\usepackage{tikz}
\usetikzlibrary{shapes,backgrounds}

\usepackage[colorlinks=true,
linkcolor=webgreen,
filecolor=webbrown,
citecolor=webgreen]{hyperref}

\definecolor{webgreen}{rgb}{0,.5,0}
\definecolor{webbrown}{rgb}{.6,0,0}

\usepackage{color}

\usepackage{graphics,amsmath,amssymb}
\usepackage{amsthm}
\usepackage{amsfonts}
\usepackage{latexsym}
\usepackage{epsf}

\newcommand{\seqnum}[1]{\href{http://oeis.org/#1}{\underline{#1}}}

\begin{document}

\theoremstyle{plain}
\newtheorem{theorem}{Theorem}
\newtheorem{corollary}[theorem]{Corollary}
\newtheorem{lemma}[theorem]{Lemma}
\newtheorem{proposition}[theorem]{Proposition}

\theoremstyle{definition}
\newtheorem{definition}[theorem]{Definition}
\newtheorem{question}[theorem]{Question}
\newtheorem{example}[theorem]{Example}
\newtheorem{conjecture}[theorem]{Conjecture}

\theoremstyle{remark}
\newtheorem{remark}[theorem]{Remark}

\title{About Some Relatives of Palindromes}

\author{Viorel Ni\c tic\u a \thanks{Department of Mathematics, West Chester
    University of Pennsylvania, West Chester, PA 19383, USA, and Institute
    of Mathematics of the Romanian Academy, P.O. Box 1--764, RO-70700
    Bucharest, Romania. vnitica@wcupa.edu} \and Andrei T\"or\"ok \thanks{
    Department of Mathematics, University of Houston, Houston, TX
    77204-3008, USA, and Institute of Mathematics of the Romanian Academy,
    P.O. Box 1--764, RO-70700 Bucharest, Romania. torok@math.uh.edu. AT
    thanks the NSF for partial support on NSF-DMS Grant 1816315.} }

\maketitle

\vskip .2 in

\begin{abstract} We introduce two new classes of integers. The first class consists of numbers $N$ for which there exists at least one nonnegative integer $A$, such that the sum of $A$ and the sum of digits of $N$,  added to the reversal of the sum, gives $N$. The second class consists of numbers $N$ for which there exists at least one nonnegative integer $A$, such that the sum of $A$ and the sum of the digits of $N$, multiplied by the reversal of the sum, gives $N$. All palindromes that either have an even number of digits or an odd number of digits and the middle digit even belong to the first class, and all squares of palindromes with at least two digits belong to the second class. These classes contain and are strictly larger than the classes of $b$-ARH numbers, respectively $b$-MRH numbers introduced in Ni\c tic\u a \cite{N1}.
\end{abstract}

\section{Introduction}\label{sec:1}

Let $b\ge 2$ be a numeration base. In Ni\c tic\u a \cite{N1}, motivated by a property of the taxicab number, 1729 \cite{Hardy}, we introduce the classes of {\it $b$-additive Ramanujan-Hardy (or $b$-ARH) numbers} and {\it $b$-multiplicative Ramanujan-Hardy (or $b$-MRH) numbers}. The first
 class consists of numbers $N$ for which there exists at least one integer $M$, called \emph{additive multiplier}, such that the product of $M$ and the sum of base $b$ digits of $N$,  added to the reversal of the product, gives $N$. The second class consists of numbers $N$ for which there exists at least one integer $M$, called \emph{multiplicative multiplier}, such that the product of $M$ and the sum of base $b$ digits of $N$,  multiplied by the reversal of the product, gives $N$. We show in \cite{N1, N3} the existence of infinite sets of $b$-ARH and $b$-MRH numbers and infinite sets of multipliers for an infinity of numeration bases. Nevertheless, several questions asked in \cite{N1, N3} remain open. In particular we would like to find obstructions to the existence of multipliers and infinite sets of multipliers of fixed multiplicity.

In this paper we change the definitions above. We replace the product between the sum of digits and the multiplier by the sum of the sum of digits and a positive extra term. This gives two new classes of numbers, $b$-wARH and $b$-wMRH. These are strictly larger than those above. We believe that the study of these classes will bring some insight into the remaining open questions in \cite{N1, N3}. Another motivation for the study of these classes of numbers is the study of numerical palindromes. All palindromes that either have an even number of digits or an odd number of digits and the middle digit even belong to the first class, and all squares of palindromes with al least two digits belong to the second class. The results in \cite{N1, N3} also give new examples of $b$-Niven numbers. These are numbers divisible by the sum of their base $b$ digits.
See, for example, \cite{CK,KD,G,FILS,N2}). In particular, any $b$-MRH number is a $b$-Niven number. We expect the study here to shine new facets of this notion.

A computer search produced many $wARH$ numbers. There are 77 integers less than 10000 having this property; see the sequence  \seqnum{A305131} in the OEIS \cite{online} and Table \ref{t:1} in this paper. For example, $121212$ has extra term $60597$. The sum of the digits is $9$, one has $9+60597=60606$, and $60606+60606=121212$.

A computer search also produced many $wMRH$ numbers. There are 365 integers less than 10000 having the property; see the sequence \seqnum{A306830}  in the OEIS \cite{online} and Table \ref{t:2} in this paper. For example, 2268 has extra term 18. The sum of the digits is $18$, one has $18+ 18=36$, and $36\times 63=2268$. 

The paper is dedicated to the study of these classes of numbers. As a by-product we also clarify some relationships between the classes of numbers introduced here and in \cite{N1}, and the well studied class of $b$-Niven numbers. The Venn diagrams in Figure 1, in which the universal set is the set of integers,  record some relationships and  lead to some open questions. The inclusion of the set of $b$-ARH numbers into the set of $b$-wARH numbers is proved in Proposition \ref{prop:1} and the inclusion of the set of $b$-MRH numbers into the set of $b$-wMRH numbers is proved in Proposition \ref{prop:3}. We believe that each proper subset in the Venn diagrams contains an infinity of integers. Those subsets for which we already know this fact are marked by a full black dot. For the others, the question is open. See Corollary\ref{cor:wARHnotNiven2} for an infinity of $b$-wARH numbers that are not $b$-Niven numbers. No large prime number can be either $b$-Niven or $b$-wMRH numbers. See the proof of Proposition \ref{prop:wmrh-notmrh} for an infinity of $b$-wMRH numbers that are not $b$-Niven numbers, and consequently neither $b$-MRH numbers.

\vskip .2in 

\begin{figure}
\centering
\begin{tikzpicture} 
\draw (0,0) rectangle (5,5);
\draw (2.2,2.2) circle (1.2cm);
\draw (2,2) circle (1.8cm);
\draw (3.5,3.5) circle (1.3cm);
\node at (3.4,4.7) {\small{b-Niven}};
\node at (1.8,3.7) {\small{b-wARH}};
\node at (1.9,1.2) {\small{b-ARH}};
\node at (1,1.2) {$\bullet$};

\draw (10,0) rectangle (15,5);
\draw (12.2,2.2) circle (1cm);
\draw (12,2) circle (1.8cm);
\draw (12.8,2.8) circle (1.9cm);
\node at (13.4,4.7) {\small{b-Niven}};
\node at (11.8,3.7) {\small{b-wMRH}};
\node at (12.5,1.2) {\small{b-MRH}};
\node at (14.5,1.2) {$\bullet$};
\node at (11,1.2) {$\bullet$};
\end{tikzpicture}
\caption{}
\end{figure}

\section{Statements of the main results}\label{sec:1-bis}

In what follows let $b\ge 2$ be an arbitrary numeration base.

\begin{definition} If $N$ is a positive integer written in base $b$, we call \emph{reversal} of $N$ and let $N^R$ denote the integer obtained from $N$ by writing its digits in reverse order.
\end{definition}

We observe that addition and multiplication of integers are independent of the numeration base. The operation of taking the reversal is not.

Let $s_b(N)$ denote the sum of the base $b$ digits of an integer $N$.

\begin{definition}\label{def:1} A positive integer $N$ written in base $b$ is called \emph{$b$-weak Ramanujan-Hardy number,} or simply {\it $b$-wARH number}, if there exists an integer $A\ge 0$, called \emph{additive extra term}, such that
\begin{equation}\label{eq:1}
N=A+s_b(N)+(A+s_b(N))^R,
\end{equation}
where $(A+s_b(N))^R$ is the reversal of base $b$-representation of $A+s_b(N)$.
\end{definition}

\begin{definition}\label{def:2} A positive integer $N$ written in base $b$ is called \emph{$b$-weak-multiplicative Ramanujan-Hardy number,} or simply {\it $b$-wMRH number}, if there exists an integer $A\ge 0$, called \emph{multiplicative extra term}, such that
\begin{equation}\label{eq:2}
N=(A+s_b(N))\cdot (A+s_b(N))^R,
\end{equation}
where $(A+s_b(N))^R$ is the reversal of base $b$-representation of $A+s_b(N)$.
\end{definition}

To simplify the notation, let $s(N)$, wARH, wMRH denote $s_{10}(N)$, 10-wARH, 10-wMRH.

We observe that the notions of $b$-wARH and $b$-wMRH numbers are dependent on the base.

\begin{example} The number $[12]_{10}$ is an wARH number with extra term $A=3$, but $[12]_{3}$ is not a $3$-wARH number. The number $[152]_{10}=21\cdot 12$ is an wMRH number with extra term $A=3$, and $[252]_3=5\cdot 7$ is a $3$-wMRH nmber with 
extra term $A=3$ but $[252]_{4}$  is not a $4$-wMRH number.
\end{example}

Once these notions are introduced and examples of such numbers found, several natural questions arise. Do there exist infinitely many $b$-wARH numbers? Do there exist infinitely many $b$-wMRH numbers?   Do there exist infinitely many additive extra terms? Do there exist infinitely many multiplicative extra terms? All these questions are positively answered below for all numeration bases.

In what follows, if $x$ is a string of digits, we let $(x)^{\land k}$ denote the string obtained by repeating $x$ $k$-times. We also let $[x]_b$ denote the value of the string $x$ in base $b$.

The following proposition is of independent interest and it is also needed later.

\begin{proposition}\label{lem:1} Let $N$ be a base $b$ integer. Then:

a) $2s_b(N)\le N,$  if $N$ has at least two digits;

b) $2s_b(N)+b-1\le N\cdot b+\frac{b-1}{2},$ if $N$ has at least two digits;

c)  If $N$ has at least three digits, then
\begin{equation}\label{eq:square}
s_b(N^2)\le N.
\end{equation}
\end{proposition}

The Proof of proposition \ref{lem:1} is done in Section \ref{sec:proof-lem1}

\begin{remark} In Proposition \ref{lem:1}, c), the condition that $N$ has at least 3 digits is necessary, as shown by $N=[13]_{11}=14_{10}$.
Indeed, $N^2=[169]_{11}$ and $s_{11}(N^2)=16>14$.
\end{remark}

The following proposition gives many examples of $b$-wARH numbers.

\begin{proposition}\label{prop:1} a) Let $N$ be a base $b$ palindrome either with an even number of digits or with an odd number of digits and the middle digit even. Then $N$ is a $b$-wARH number.

b) Let $N$ be a $b$-ARH number, Then $N$ is a $b$-wARH number.
\end{proposition}

\begin{remark} We observe that \cite[Theorem1]{N3} gives for any $b\ge2$ an infinity of $b$-wARH number  that are not palindromes.
\end{remark}

\begin{corollary} For any string of digits $I$ there exists an infinity of $b$-wARH numbers that contain $I$ in their base $b$-representation. 
\end{corollary}

\begin{proof} The string $I$ is part of an infinity of base $b$ palindromes with an even number of digits.
\end{proof}

\begin{corollary} For any integer $N$ there exists an infinity of integers $M$ such that $N\cdot M$ is a $b$-wARH number. Consequently, all integers are divisors of $b$-wARH numbers.
\end{corollary}

\begin{proof} It is proved in \cite[Theorem 5]{N3} that for any integer $N$ there exists an infinity of integers $M$ such that $N\cdot M$ is a palindrome. If the palindrome has an even number of digits, we are done. Otherwise, if $P=N\cdot M$ is an arbitrary palindrome with $k$ digits, consider the product $P\cdot [1(0)^{\land k-1}1]_b$, which is a palindrome with $2k$ digits.
\end{proof}

\begin{corollary} For any $b\ge 2$ there exist an infinity of arithmetic progressions of length $b$ of $b$-wARH numbers.
\end{corollary}

\begin{proof} If $I$ is a string of base $b$-digits of length at least 1, consider the following arithmetic progression of palindromes:
\begin{equation*}
[I00I^R]_b, [I11I^R]_b,[I22I^R]_b,[I33I^R]_b,\dots \dots,[I(b-2)(b-2)I^R]_b,[I(b-1)( b-1)I]_b.
\end{equation*}
\end{proof}

\begin{corollary}\label{cor:wARHnotNiven2} There exists an infinity of $b$-wARH numbers that are not $b$-Niven numbers.
\end{corollary}

\begin{proof} For any $k\ge 1$ define $N_k=[1(0)^{\land k}(b-1(b-1)(0)^{\land k}1]_b$. Then $s_b(N_k)=2b$ and $N_k$ is not divisible by b. But $N_k$ are palindromes with even number of digits, so they are $b$-wARH numbers.
\end{proof}

We show in \cite[Theorem 26]{N1} the existence of an infinity of integers that are not $b$-ARH. The following result has a similar proof.

\begin{proposition}\label{prop:notwARH} There exists an infinity of numbers that are not $b$-wARH numbers.
\end{proposition}

The following result complements \cite[Corollary 19]{N1}, which applies only for $b$ even and gives an infinity of $b$-ARH numbers that are not $b$-MRH numbers.

\begin{proposition}\label{prop:arhnotmrh} There exists an infinity of $b$-wARH numbers that are not $b$-MRH numbers.
\end{proposition}

Proposition \ref{prop:arhnotmrh} is proved in Section \ref{sec:arhnotmrh}.

\begin{question} Does there exist an infinity of $b$-wARH numbers that are not $b$-ARH numbers?
\end{question}

\begin{proposition}\label{prop:p2} For any $b\ge 2$ there exists an infinity of $b$-wARH numbers and an infinity of extra terms.
\end{proposition}

\begin{proof} Consider the sequence of palindromes  $N_k=[1(0)^{\land k}(0)^{\land k}1]_b, k\ge 1,$ with additive terms $A_k=b^{2k}-2$.
\end{proof}

The following proposition gives many examples of $b$-wMRH numbers.

\begin{proposition}\label{prop:3} a) Let $P$ be a a base $b$-palindrome with at least two digits and let $N=P^2$. Then $N$ is a $b$-wMRH number.
b) Let $N$ be a $b$-MRH number, Then $N$ is a $b$-wMRH number.
\end{proposition}

\begin{remark} We observe that \cite[Theorem4]{N3} gives for an infinity of numeration bases an infinity of $b$-wMRH number  that are not squares of palindromes.
\end{remark}

\begin{corollary} For any string of base $b$ digits $I$ there exists an infinity of $b$-wMRH numbers that contain $I$ in their base $b$-representation. 
\end{corollary}

\begin{proof} It is enough to show that the string $I$ is part of an infinity of base $b$ squares of base $b$ palindromes. If $[I]_b$ is even,
let $[J]_b$ be a $k_0$ digit string such that $2J=I$. Then $I$ is part of the base $b$-representation of 
$\left ([1(0)^{\land k}J(0)^{\land k}1]_b\right )^2,$ for all $k\ge 3k_0.$ If $[I]_b$ is odd, let $[J]_b$ be a $k_0$ digit string such that $2J+1=I$. Then $I$ is part of the base $b$-representation of $\left (\left [J(0)^{\land k}1(0)^{\land k}1(0)^{\land k}J\right ]_b\right )^2$ for all $k\ge 3k_0.$
\end{proof}

\begin{corollary} For any integer $N$ there exists an infinity of integers $M$ such that $N\cdot M$ is a $b$-wMRH number. Consequently, all integers are divisors of $b$-wMRH numbers.
\end{corollary}

\begin{proof} It is proved in \cite[Theorem 5]{N3} that for any integer $N$ there exists an infinity of integers $M$ such that $N\cdot M$ is a palindrome. Then the product $N\cdot M\cdot (N\cdot M)$ is a $b$-wMRH number.
\end{proof}

It is well known that there exists an infinity of numbers that are not $b$-Niven. As a $b$-MRH number is $b$-Niven, this gives an infinity of numbers that are not $b$-MRH numbers.

\begin{proposition} There exists an infinity of numbers that are not $b$-wMRH numbers.
\end{proposition}

\begin{proof} No prime number is $b$-wMRH number.
\end{proof}

\begin{remark} The condition in Proposition \ref{prop:3} that $P$ has at least 2 digits is necessary. Some squares of one digit numbers are $b$-wMRH number, for example 81, and some are not, for example 25.
\end{remark}

\begin{proposition}\label{prop:p2} For any $b\ge 2$ there exists an infinity of $b$-wMRH numbers and an infinity of extra terms.
\end{proposition}

\begin{proof} Consider the sequence $N_k=\left ( [1(0)^{\land {k-1}}1]_b\right )^2, k\ge 1,$ with additive terms $A_k=b^k-1$.
\end{proof}

Combining Proposition \ref{prop:1}, c) and \cite[Theorems 13,15 ]{N1} one has the following result.

\begin{theorem}\label{thm:growth} a) Consider the numbers
\begin{equation}\label{eq:3**}
N_k=[(1)^{\land k}]_b,
\end{equation}
where $b$ is even, $k=[1(0)^{\land p}]_b, p\ge 1$, $p$ an arbitrary natural number. All numbers $N_k$ are $b$-wARH numbers. Each $N_k$ has a subset of additive multipliers of cardinality $2^{\frac{k-2p}{2}}$ consisting of all integers $k\cdot\left ( [(1)^{\land p}I]_b\right )$, where $I$ is a sequence of $0$ and $1$ of length $k-2p$ in which no two digits symmetric about the center of the sequence are identical.

b) Consider the numbers
\begin{equation}\label{eq:3larger-growth}
N_k=[(1)^{\land p}(10)^{\land k-2p}0(1)^{\land p}]_b,
\end{equation}
where $b$ is even and $k=[1(0)^{\land p}]_b, p\ge 1,$ $p$ arbitrary natural number. All numbers $N_k$ are $b$-wARH numbers.
For each $N_k$ the set of additive extra terms has cardinality $(b-1)^\frac{k-2p}{2}$ and consists of all integers $2\cdot \left ([(1)^{\land p}I0]_b-1\right )$, where $I$ is a concatenation of $k-2p$ two digits strings of type $0\alpha, \alpha\not = 0$, in which any pair of nonzero digits symmetric about the center of $I0$ have their sum equal to $b$.
\end{theorem}

\begin{corollary} If $b$ is even, there exists infinitely many $b$-wARH numbers that have at least two extra terms.
\end{corollary}

\begin{question}\label{q:22-new} Do there exist infinitely many $b$-wMRH numbers that have at least two extra terms?
\end{question}

\begin{proposition}\label{prop:wmrh-notmrh} There exists an infinity of $b$-wMRH numbers that are not $b$-MRH numbers.
\end{proposition}

\begin{question} Does there exist an infinitely of $b$-wARH numbers that are not $b$-wMRH?
\end{question}

Motivated by the results in Theorem \ref{thm:growth}, we introduce the following notions.

\begin{definition} If $N$ is a $b$-wARH number, let the \emph{multiplicity} of $N$ be the cardinality of the corresponding set of additive extra terms.
\end{definition}

\begin{definition} If $N$ is a $b$-wMRH number, let the \emph{multiplicity} of $N$ be the cardinality of the corresponding set of multiplicative extra terms.
\end{definition}

Theorem \ref{thm:growth} has the following corollary.

\begin{corollary} The multiplicity of $b$-wARH numbers is unbounded for any even base.
\end{corollary}

\begin{question}\label{q:13} Is the multiplicity of $b$-wMRH numbers bounded?
\end{question}

We show in \cite[Theorem 25]{N1}  an infinity of $b$-Niven numbers that are not $b$-MRH numbers. The following question is open.

\begin{question} Does there exist an infinity of $b$-Niven numbers that are not $b$-wMRH numbers?
\end{question}

We show in Section \ref{sec:8} that $2$ is not a multiplicative extra term for base $10$. We do not know how to answer the following questions for any base.

\begin{question}\label{q:6} Do there exist infinitely many integers that are not additive extra terms?
\end{question}

\begin{question}\label{q:7} Do there exist infinitely many integers that are not multiplicative extra terms?
\end{question}

In what follows let $\lfloor x \rfloor$ denote the integer part, let $\ln x$ denote the natural logarithm and let $\log_b x$ denote base $b$ logarithm of the positive real number $x$.

The following theorems give bounds for the number of digits in a $b$-wARH number with fixed extra term. Due to independent interest and in order to simplify the statements of other results we consider first the case when the extra term is $A=0$.

\begin{theorem}\label{thm:0-wARH} Let $N$ be a $b$-wARH number with $k$ digits and  additive exta term $A=0$. Then $N=0$, $N=[11]_2$, $N=[22]_3$, or $N=[1(b-2)]_b$.
\end{theorem}

\begin{remark} We leave as open the problem of finding all $b$-wMRH numbers with extra term $A=0$.
We only observe that if $b=10$ a $wMRH$ number with $A=0$ is also an $MRH$ number with multiplier $M=1$. It is shown in \cite{N1}
that all such numbers are 1, 81, 1458 and 1729.
\end{remark}

\begin{theorem}\label{thm:1-wARH} Let $N$ be a $b$-wARH number with $k$ digits and  additive exta term $A$. Then
\begin{displaymath}
	k\le A+4.
\end{displaymath}
\end{theorem}

\begin{corollary} For fixed additive extra term $A$ and base $b$, the set of $b$-wARH numbers with extra term $A$ is finite.
\end{corollary}

\begin{theorem}\label{thm:2-strong} Let $N$ be a $b$-wARH number with $k$ digits and additive extra term $A$. Under the assumption $A\ge b^3$ one has:
\begin{equation}\label{eq:4-Strong}
\begin{gathered}
k\le 2\lfloor \log_b A \rfloor.
\end{gathered}
\end{equation}
\end{theorem}

The following theorems give bounds for the number of digits in a $b$-wMRH number with fixed extra term.

\begin{theorem}\label{thm:3} Let $N$ be a $b$-wMRH number with $k$ digits and multiplicative extra term $A\ge 1$. Then
\begin{displaymath}
k\leq \begin{cases}
	A+4, & \text{ if } b\ge 6;\\
	A+5, & \text{ if } 2\le b\le 5.
\end{cases}
\end{displaymath}
\end{theorem}

\begin{corollary} For fixed multiplicative extra terms $A$ and base $b$, the set of $b$-wMRH numbers with extra term $A$ is finite.
\end{corollary}

\begin{theorem}\label{thm:3-new} Let $N$ be a $b$-wMRH number with $k$ digits and multiplicative extra term $A\ge 1$. Under any of the following assumptions:
\begin{itemize}
\item $b \ge 3$ and $A\ge b^3;$
\item $b=2$ and $A\ge b^{2};$
\end{itemize}
one has
\begin{equation}\label{eq:5-new}
k\le 3\lfloor \log_b A\rfloor.
\end{equation}
\end{theorem}

We summarize the rest of the paper. Proposition \ref{lem:1} is proved in Section \ref{sec:proof-lem1}, Proposition \ref{prop:1} is proved in Section \ref{sec:proof-prop1}, Proposition  \ref{prop:arhnotmrh} is proved in Section \ref{sec:arhnotmrh}, Proposition \ref{prop:3} is proved in Section \ref{sec:proof-prop3}, Proposition \ref{prop:wmrh-notmrh} is proved in Section \ref{sec:some-sec11}, Proposition \ref{thm:0-wARH} is proved in Section \ref{sec:6-0}, Theorem \ref{thm:1-wARH} is proved in Section \ref{sec:6}, Theorem \ref{thm:2-strong} is proved in Section \ref{sec-6-bifore}, Theorem \ref{thm:3} is proved in Section \ref{sec:6-biss}, and Theorem \ref{thm:3-new} is proved in Section \ref{secproofthm3}. In Section \ref{sec:8} we show examples of wARH numbers and ask additional questions and in Section \ref{sec:8-bis} we show examples of wMRH numbers and ask additional questions.

\section{Proof of Proposition \ref{lem:1}}\label{sec:proof-lem1}

\begin{proof} a), b) Clearly b) implies a), so it is enough to prove b). Assume $N$ has $n\ge 2$ digits.  Then $N\ge b^{n-1}$ and $s_b(N)\le n(b-1)$.
To finish the proof, we show by induction on $n\ge 2$ that
\begin{equation}\label{eq:2}
 2(b-1)n+(b-1)\le b\cdot(b^{n-1})+\frac{b-1}{2}.
\end{equation}

Inequality \eqref{eq:2} is true if $n=2$. Assume now that it is true for $n$ and prove it for $n+1$. Induction hypothesis gives that:
\begin{equation}
2(b-1)(n+1)+(b-1)=2(b-1)n+2(b-1)+(b-1)\le b\cdot(b^{n-1})+\frac{b-1}{2}+2(b-1).
\end{equation}

We still need to show that:
\begin{equation}\label{eq:3}
b\cdot(b^{n-1})+\frac{b-1}{2}+2(b-1)\le  b\cdot(b^{n})+\frac{b-1}{2}.
\end{equation}
After some cancellation, equation \eqref{eq:3} becomes $2\le b^n$, which is true for $n\ge2, b\ge2$.

c) Assume that $N$ has $n\ge 3$ digits. Then $b^{n-1}\le N\le b^n-1$. Hence
\begin{equation}
b^{2n-2}\le N^2\ge (b^n-1)^2=b^{2n}-2b^n+1.
\end{equation}
So $N^2$ has $2n-1$ digits, and$s_b(N^2)\le (b-1)(2n-1)$. To finish the proof it is enough to show that
\begin{equation}\label{eq:1}
 (b-1)(2n-1)\le b^n-1.
\end{equation}

Equation \eqref{eq:1} is true for $n=3$ and $b\ge 2$. We assume $n\ge 4$ fixed and prove \eqref{eq:1} by induction on $b\ge 3$. The induction hypothesis, $b\ge 3$, and the binomial expansion of $(1+b)^n$, imply that for all $b\ge 3$ one has that:
\begin{equation*}
b(2n-1)=(b-1)(2n-1)+(2b-1)\le b^n-1+(2n-1)\le (b+1)^{n-1},
\end{equation*}
which finishes the proof of \eqref{eq:1} if $b\ge 2$.

If $b=2$ Inequality \eqref{eq:1} becomes $2n-1\le 2^n-1,$ true for $n\ge 4.$ There are only 4 integers with $b=2, n=3,$ and for them
inequality \eqref{eq:square} can be checked numerically.
\end{proof}

\section{Proof of Proposition \ref{prop:1}}\label{sec:proof-prop1}

\begin{proof} a) Assume first that $N=[a_1a_2\ldots a_na_n\ldots a_2a_1]_b$. Define $A=[a_1a_2\ldots a_n(0)^{\land n}]_b-s_b(N).$ Then $A\ge 0$ due to Proposition \ref{lem:1} a) applied to $[a_1a_2\ldots a_n(0^{\land n}]_b.$ One has that:

\begin{equation*}
\begin{gathered}
(s_b(N)+A)+(s_b(N)+A)^R=[a_1a_2\ldots a_n(0)^{\land n}]_b+([a_1a_2\ldots a_n(0)^{\land n}]_b)^R\\
=[a_1a_2\ldots a_n(0)^{\land n}]_b+[a_na_{n-1}\ldots a_1]_b=N.
\end{gathered}
\end{equation*}

Now assume that $N=[a_1a_2\ldots a_na_{n+1}a_n\ldots a_2a_1]_b$, where $a_{n+1}$ is even. Define $A=[a_1a_2\ldots a_n\left( \frac{a_{n+1}}{2}\right )(0)^{\land n}]_b-s_b(N).$ Then $A\ge 0$ due to Lemma \ref{lem:1} a) applied to $[a_1a_2\ldots a_n\left( \frac{a_{n+1}}{2}\right )(0)^{\land n}]_b.$ One has that:

\begin{equation*}
\begin{gathered}
(s_b(N)+A)+(s_b(N)+A)^R=[a_1a_2\ldots a_n(0)^{\land n}]_b+([a_1a_2\ldots a_n(0)^{\land n}]_b)^R\\
=[a_1a_2\ldots a_n\left( \frac{a_{n+1}}{2}\right )(0)^{\land n}]_b+[\left( \frac{a_{n+1}}{2}\right )a_na_{n-1}\ldots a_1]_b=N.
\end{gathered}
\end{equation*}

b) Let $N$ be a $b$-ARH number with additive multiplier $M\ge 1$. Then $N$ is also a $b$-wARH number with extra term $A=s_b(N)(M-1)$.
\end{proof}

\section{Proof of Proposition \ref{prop:arhnotmrh}}\label{sec:arhnotmrh}

\begin{proof} It is known that a base $b$ number is divisible by $b-1$ only if and only if the sum of its digits is divisible by $b-1$.
Consider the palindromes
\begin{equation*}
N_k=[(b-1)(0)^{\land k}(b-1)]_b, k \text{ even}.
\end{equation*}

It follows from Proposition \ref{prop:1}, a), that the numbers $N_k$ are $b-wARH$ numbers. If $b=2$, then $s_b(N_k)=2$, but $N_k$ is odd, so $N_k$ is not a $b$-MRH number. Assume $b\ge 4$. As $s_b(N)=2(b-1)$ it follows that $N_k$
is divisible by $b-1$, but not by $(b-1)^2$. Nevertheless, if $N_k$ is $b$-MRH number then it must be divisible by $(b-1)^2$.
If $b=3$ consider the palindromes $N_k=[2(0)^{\land k}2(0)^{\land k}2]_3$. It follows from Proposition \ref{prop:1}, a), that the numbers
$N_k$ are $3-wARH$ numbers. As $s_3(N_k)=6$ and $N_k$ are divisible by 2, but not by 4, it follows that $N_k$ are not $3-MRH$ numbers.
\end{proof}

\section{Proof of Proposition \ref{prop:3}}\label{sec:proof-prop3}

\begin{proof} a) Let $P$ be a base $b$ palindrome and let $N=P^2$. Assume that $P$ has at least three digits. It follows from Proposition \ref{lem:1} c), that $s_b(N)\le P$. Let $A=P-s_b(N)$. Then $N$ is a $b$-wMRH number with extra term $A$. Assume now that $P$ has two digits. Then $P=[aa]_b$ for $1\le a\le b-1$. We will show that formula \eqref{eq:square} is still valid. Then the argument above can be applied again. We distinguish three cases.

{\em Case 1.} $2a^2<b$ Then $P=a(b+1),$ $N=[a^2(2a^2)a^2]_b,$ and $s_b(N)=4a^2$. If $a>1$ one has that:
\begin{equation*}
s_b(N)=4a^2<4\cdot \frac{b}{2}=2b<a(b+1)=P.
\end{equation*}

If $a=1$ and $b\ge 3$ one has that:
\begin{equation*}
s_b(N)=4\le b+1=P.
\end{equation*}

If $a=1$ and $b=2$ then the condition $2a^2<b$ is not satisfied.

{\em Case 2.} $a^2<b\le 2a^2$ We distinguish two subcases: $a) a^2+1<b$ and $b) a^2+1=b$.

{\em Subcase a)} $s_b(N)=a^2+1+2a^2-b+a^2=4a^2+1-b<3(b-1)$. If $a\ge 3$ then
\begin{equation*}
s_b(N)<3(b-1)<a(b+1)=P.
\end{equation*}

If $a=1$, the condition $b\le 2a^2$ implies that $b=2$. In this case $P=[11]_2$ and
\begin{equation*}
s_b(P^2)=s_b([10001]_2)=2\le P=3.
\end{equation*}

If $a=2$, $b\in\{6,7,8\}$. So $P=[22]_6, P=[22]_7$ or $P=[22]_8$. These cases can be checked numerically.

{\em Subcase b)} $s_b(N)=a^2+1+2a^2-b+a^2=3a^2=3(b-1).$ If $a\ge 3$ then
\begin{equation*}
s_b(N)=3(b-1)\le a(b+1)=P.
\end{equation*}

If $a=1$ then $b=2$ and $P=[11]_2$. If $a=2$ then $5<b<8$ and the only new possibility is $[22]_5$ which can can be checked numerically.

{\em Case 3.} $a^2\ge b$ Note that each "carry over" in the computation of $P^2$ reduces $s_b(P^2)$ by $b$ and also increases it by 1.
We have at least 4 carry overs, so the largest value for $s_b(P^2)$ is $4a^2-4b+4$. The inequality $s_b(P^2)\le P$ becomes
\begin{equation*}
4a^2-4b+4\le a(b+1),
\end{equation*}
or equivalently
\begin{equation}\label{eq:quadr}
4a^2-a(b+1)+4(1-b)\le 0, \text{ for } 1\le a\le b-1.
\end{equation}

If $b\ge 3$, the quadratic function in \eqref{eq:quadr} has the vertex at $a=\frac{b+1}{2}\in (1,b-1)$, so its largest values in the interval [1,b-1] are reached in the endpoints. Since its value in $a=1$ is $7-5b$ and its value in $a=b-1$ is $6-7b$, it follows that \eqref{eq:quadr} holds. If $b=2$ the remaining case is $P=[11]_2$.

b) Let $N$ be a $b$-MRH number with additive multiplier $M\ge 1$. Then $N$ is a $b$-wMRH number with extra term $A=s_b(N)(M-1)$.
\end{proof}

\section{Proof of Proposition \ref{prop:wmrh-notmrh}}\label{sec:some-sec11}

\begin{proof} It follows from Proposition \ref{prop:3} that it is enough to find an infinity of squares of palindromes that are not $b$-Niven numbers.

If $b=2$ consider $N_k=\left ( [1(0)^{\land k}1(0)^{\land k}1]_2\right )^2=[1(0)^{\land k-1}1(0)^{\land k-1}11(0)^{\land k-1}1(0)^{\land k+1}1]_2.$ Then $s_b(N_k)=6$ and $N_k$ is not divisible by 2 because it is odd. If $b$ is even, and $b\not =2$, then consider $N_k =\left ([1(0)^{\land k}1]_b\right )^2=[1(0)^{\land k}2(0)^{\land k}1]_b.$ Then $s_b(N_k)=4$ and $N_k$ is not divisible by 2 because it is odd.

If $b$ is odd and $b$ congruent to 0 or 2 modulo 3, consider the numbers
\begin{equation*}
\begin{gathered}
N_k=\left ( [1(0)^{\land k}1(0)^{\land k}1]_b\right )^2\\
=[1(0)^{\land k}2(0)^{\land k}3(0)^{\land k}2(0)^{\land k}1]_b. k+1\text{ odd.}
\end{gathered}
\end{equation*}
Then $s_b(N_k)=9$ and $N_k$ is not divisible by 3 because $[1(0)^{\land k}1(0)^{\land k}1]_b$ is not divisible by 3. For the case, $b\ge 11$ congruent to 1 modulo 3, consider the numbers
\begin{equation*}
\begin{gathered}
N_k=\left ( [2(0)^{\land k}1(0)^{\land k}2]_b\right )^2\\
=[4(0)^{\land k}3(0)^{\land k}(10)(0)^{\land k}3(0)^{\land k}4]_b. k+1.
\end{gathered}
\end{equation*}
Then $s_b(N_k)=24$ and $N_k$ is not divisible by 3 because $[2(0)^{\land k}1(0)^{\land k}2]_b$ is not divisible by 3. If $b\le 11$, then $b\in\{9,7,5,3\}$ and these cases are covered above.

\end{proof}

\section{Proof of Theorem \ref{thm:0-wARH}}\label{sec:6-0}

Let $N\ge 1$ be a $b$-wARH number with extra term $A=0$ and $k$ digits. Then $N$ is also a $b$-ARH number with additive multiplier $M=1$. It follows from \cite[Theorem 35]{N1} that $k\le 2$ if $b\ge 4$ and $k\le 3$ if $b=2$ or $b=3$. If $k=1$ and $N>0,$ then $s_b(N)+s_b(N)^R>N$, so we can assume $k\ge 2$. If $k=2$, then $N=[\alpha\beta]_b$ with $1\le \alpha, \beta\le b-1$. If $\alpha+\beta<b$, then the equation $s_b(N)+s_b(N)^R=N$ gives $\alpha(b-2)=\beta\le b-1$, which implies $\alpha\le 2$. If $\alpha=0$, then $\beta=0$, so $N=0$. 
If $\alpha=1$, then $\beta=b-2$ and $N=[1(b-2)]_2$. If $\alpha=2$ then $b=3$ and $\beta=2$, so $N=[22]_3$.
Assume now $\alpha+\beta\ge b$. Then $\alpha b+\beta=2(1+\alpha+\beta-b)$ which implies $2(b-2)\le 2+\beta-b\le 1$. So $\alpha=1$ and $b=2$, which implies $\beta=1$. So $N=[11]_2.$  The remaining cases with $k=3$ and $a=2, a=3$ are finite in number and do not give any other $b$-wARH number. 

\section{Proof of Theorem \ref{thm:1-wARH}}\label{sec:6}

The case $A=0$ is covered by Theorem \ref{thm:0-wARH}. Assume that $N$ is a $b$-wARH number with $k\ge 2$ digits and additive extra term $A\ge 1$. One has that:
\begin{equation}\label{eq:but1}
b^{k-1}\le N=(s_b(N)+A)+(s_b(N)+A)^R\le (b+1)\left ( (b-1)k+A\right ).
\end{equation}

We show by induction on $k$ that:
\begin{equation}\label{eq:but2}
(b+1)\left ( (b-1)k+A\right )<b^{k-1}, \text{ for }k\ge A+5, b\ge 2, A\ge 1.
\end{equation}
As \eqref{eq:but1} and \eqref{eq:but2} are contradictory, this finishes the proof of the theorem.

For $k=A+5$, \eqref{eq:but2} gives that:
\begin{equation}\label{eq:but3}
(b+1)\left ( (b-1)(A+5)+A\right )<b^{A+4}, b\ge 2, A\ge 1,
\end{equation}
which we prove by induction on $A$. If $A=1$, \eqref{eq:but3} gives that $(b+1)\left ( 6(b-1)+1\right )<b^5,$ which is true for $b\ge 2$.

We show the induction step in \eqref{eq:but3}. From the induction hypothesis one has that:
\begin{equation*}
b^{A+5}=b^{A+4}b\ge b(b+1)\left ( (b-1)(A+5)+A\right ).
\end{equation*}
One still needs to show that
\begin{equation*}
b(b+1)\left ( (b-1)(A+5)+A\right )\ge (b+1)\left ( (b-1)(A+6)+A+1\right ).
\end{equation*}
The last inequality follows from $b(A+5)\ge A+6$ and $bA\ge A+1$.

We show the induction step in \eqref{eq:but2}. From the induction hypothesis one has that:
\begin{equation*}
b^{k}=b^{k-1}b\ge b(b+1)\left ( (b-1)k+A\right  ).
\end{equation*}
One still needs to show that
\begin{equation*}
b(b+1)\left ( (b-1)k+A\right  ) \ge (b+1)\left ( (b-1)(k+1)+A\right  ).
\end{equation*}
Last inequality is equivalent to
\begin{equation*}
b(b-1)k+bA \ge (b-1)(k+1)+A,
\end{equation*}
which follows due to $bk\ge k+1$ and $b\ge 1$.

\section{Proof of Theorem \ref{thm:2-strong}}\label{sec-6-bifore}

\begin{proof} Assume that $N$ is a $b$-wARH number with $k\ge 2$ digits and additive extra term $A\ge 1$. One has \eqref{eq:but1}.
We show by induction on $k$ that
\begin{equation}\label{eq:neweq1}
b^{k-1}>(b+1)\left ((b-1)k+A\right ), A\ge b^3, k\ge 2\lfloor\log_bA\rfloor, b\ge 2,
\end{equation} 
which is in contradiction to \eqref{eq:but1} and finishes the proof of the theorem.

In order to prove \eqref{eq:neweq1} for $k= 2\lfloor\log_bA\rfloor$ it is enough to show that
\begin{equation}\label{eq:neweq2}
b^{2\log_b A}>(b^2-1)(2\log_b A+1)+(b-1)A, b\ge 2, A\ge b^3,
\end{equation}
which we will prove by induction on $A$. If $A=b^3$, then \eqref{eq:neweq2} becomes $b^{6}>(b^2-1)\cdot 7+(b-1)b^3$, shich is true for $b\ge 2$. we how the induction step in \eqref{eq:neweq2}. From induction hypothesis follows that
\begin{equation*}
(A+1)^2=a^2+2A+1>(b^2-1)(\log_bA^2+1)+(b-1)A+2A+1.
\end{equation*}
One still needs to check that:
\begin{equation*}
(b^2-1)(\log_bA^2+1)+(b-1)A+2A+1\ge (b^2-1)(\log_b(A+1)^2+1)+(b-1)(A+1).
\end{equation*}
Last equation is equivalent to $(b^2-1)\log_b\left (\frac{A}{A+1}\right ) +2A+1>b-1$, which is clearly true if $A\ge b^3$.

It remauns to show the induction step in \eqref{eq:neweq1}. From induction hypothesis follows that
\begin{equation*}
b^k=b\cdot b^{k-1}>(b+1)((b-1)k+A).
\end{equation*}
One still needs to show
\begin{equation*}
(b+1)((b-1)k+A)\ge (b+1)((b-1)(k+1)+A.
\end{equation*}
Last equation is equivalent to $(b-1)^2k+(b-1)A\ge b-1$, which is clearly true for $A\ge 1, b\ge 2$.

\end{proof}

\section{Proof of Theorem \ref{thm:3}}\label{sec:6-biss}

\begin{proof} Assume that $N$ is a $b$-wMRH number with $k\ge 2$ digits and additive extra term $A\ge 1$. One has that:
\begin{equation}\label{eq:but1-bis}
b^{k-1}\le N=(s_b(N)+A)\cdot(s_b(N)+A)^R\le b\left ( (b-1)k+A\right )^2.
\end{equation}

In order to prove the theorem for $b\ge 6$, one shows by induction on $k$ that:
\begin{equation}\label{new-eq11}
b\left ( (b-1)k+A\right )^2<b^{k-1}, \text{ if }k\ge A+5, A\ge 1, b\ge 6.
\end{equation}

If $k=A+5$ \eqref{new-eq11} becomes
\begin{equation}\label{new-eq11-bis}
b\left ( (b-1)(A+5)+A\right )^2<b^{A+4}.
\end{equation}
We prove \eqref{new-eq11-bis} by induction on $A\ge 1$. If $A=1$, \eqref{new-eq11-bis} becomes $b\left ( (b-1)6+1\right )^2<b^5$, which is true for $b\ge 6$. We show the induction step in \eqref{new-eq11-bis}. It follows from the induction hypothesis that\begin{equation*}
b^{A+5}=b\cdot b^{A+4}>b^2\left ( (b-1)(A+5)+A\right )^2.
\end{equation*}

One still needs to check that
\begin{equation*}
b^2\left ( (b-1)(A+5)+A\right )^2\ge b(b-1)(A+6)+A+1)^2.
\end{equation*}
Last equation is equivalent to
\begin{equation*}
\sqrt{b}(b-1)(A+5)+\sqrt{b}A\ge (b-1)(A+6)+A+1
\end{equation*}
which is clearly true if $b\ge 6$. We show the induction step in \eqref{new-eq11}. It follows from the induction hypothesis that
\begin{equation*}
b^{k}=b\cdot b^{k-1}>b^2\left ( (b-1)k+A\right )^2.
\end{equation*}
One still needs to check that
\begin{equation*}
b^2\left ( (b-1)k+A\right )^2\ge b((b-1)(k+1)+A)^2.
\end{equation*}

Last equation is equivalent to
\begin{equation*}
\sqrt{b}(b-1)k+\sqrt{b}A\ge (b-1)(k+1)+A,
\end{equation*}
which is clearly true if $b\ge 6$. 

Assume now $2\le b\le 5$. One shows by induction on $k$ that:
\begin{equation}\label{new-eq11=b5}
b\left ( (b-1)k+A\right )^2<b^{k-1}, \text{ if }k\ge A+6, A\ge 1.
\end{equation}
This finishes the proof of the theorem if $2\le b\le 5$.

If $k=A+6$ then \eqref{new-eq11=b5} becomes the following equation which is proved by induction on $A\ge 1$.
\begin{equation}
b\left ((b-1)(A+6)+A \right )<5^{A+5}, 2\le b\le 5.
\end{equation}
\end{proof}

\section{Proof of Theorem \ref{thm:3-new}}\label{secproofthm3}

\begin{proof} Assume that $N$ is a $b$-wMRH number with $k\ge 2$ digits and additive extra term $A\ge 1$. One has \eqref{eq:but1-bis}.
In order to finish the proof of the theorem in the case $b\ge 3$ one shows by induction on $k$ that
\begin{equation}\label{eq:now13}
b^{k-1}>b(b-1)\left ( (b-1)k+A\right ) \text{ for } k\ge 3\lfloor \log_b A\rfloor+1, b\ge 3, A\ge b^3.
\end{equation}
To prove \eqref{eq:now13} for $k= 3\lfloor \log_b A\rfloor+1$ it is enough to show by induction on $A$ that:
\begin{equation}\label{eq:newnew3}
b^{3\log_bA-3}>(b-1)\left ( (b-1)(3\log_bA+1)+A\right ), b\ge 3, A\ge b^2.
\end{equation}

If $A=b^3$, \eqref{eq:now13} becomes $b^6>(b-1)\left ( (b-1)\cdot 10+b^3\right ),$ which is true for $b\ge 3$.

We show the induction step in \eqref{eq:newnew3}. It follows from the induction hypothesis that
\begin{equation*}
b^{3\log_b(A+1)-3}=b^{3\log_b A-3}\cdot \left (\frac{A+1}{A}\right )^3 > \left (\frac{A+1}{A}\right )^3\cdot (b-1)\left ( (b-1)(3\log_bA+1)+A\right ).
\end{equation*}

One still needs to show
\begin{equation*}
\left (\frac{A+1}{A}\right )^3\cdot (b-1)\left ( (b-1)(3\log_bA+1)+A\right )\ge (b-1)\left ( (b-1)(3\log_b(A+1)+1)+(A+1)\right ).
\end{equation*}
The last inequality follows due to the following inequalities which are true for $A\ge b^2, b\ge 3$:
\begin{equation*}
\begin{gathered}
\left (\frac{A+1}{A}\right )^3\cdot (b-1)( (b-1)(3\log_bA+1)>(b-1)^2(3\log_b(A+1)+1),\\
\left (\frac{A+1}{A}\right )^3\cdot A>A+1.
\end{gathered}
\end{equation*}

We show the induction step in \eqref{eq:now13}. It follows from the induction hypothesis that
\begin{equation*}
b^{k}=b\cdot b^{k-1}>b(b-1)\left ( (b-1)k+A\right ).
\end{equation*}

One still needs to show
\begin{equation*}
b(b-1)\left ( (b-1)k+A\right )\ge (b-1)\left ( (b-1)(k+1)+A\right ).
\end{equation*}

Last inequality follows from the following inequalities which are obvious for $b\ge 2$:
\begin{equation*}
b(b-1)k \ge (b-1)(k+1,\ \ \ \ \ \ bA\ge A.
\end{equation*}

If $b=2$ one shows by induction on $k$ that:
\begin{equation}\label{eq:now7}
2^{k-1}>2(k+A), \text{ for } k\ge 3\lfloor\log_2 A\rfloor, A\ge 4,
\end{equation}
which is contradictory to \eqref{eq:but1-bis} and ends the proof of the theorem.

In order to prove \eqref{eq:now7} for $k=3\lfloor\log_2 A\rfloor$, it is enough to show by induction on $A$ that:
\begin{equation}\label{eq:new27}
2^{3\log_2A-1}\ge 2\left ( 3\log_2A+4\right ), A\ge 4.
\end{equation}

If $A=4$, \eqref{eq:new27} becomes $2^5\ge 12$, which is true. We show the induction step in \eqref{eq:new27}. It follows from the induction hypothesis that:
\begin{equation*}
2^{3\log_2(A+1)-1}=\left ( \frac {A+1}{A}\right )^3\cdot 2^{3\log_2A-1}\ge \left ( \frac {A+1}{A}\right)^3\cdot 2\left ( 3\log_2A+4\right ).
\end{equation*}

One still needs to show that
\begin{equation*}
\left ( \frac {A+1}{A}\right)^3\cdot 2\left ( 3\log_2A+4\right )\ge 2\left ( 3\log_2(A+1)+4\right ).
\end{equation*}
The last inequality is true for $A\ge 4$ due to $A^A\ge A+1$.

\end{proof}

\section{Examples of wARH numbers}\label{sec:8}
\begin{table}
\centering
\scalebox{0.8}{
\begin{tabular}{ |c | c | c|c|c|c|c|c|c | c | c| c|c|c|c|c|c|c|}
\hline
 $N$ & $A$ & $N$ & $A$ & $N$ & $A$ & $N$ & $A$ & $N$ & $A$& $N$ & $A$ & $N$ & $A$& $N$ & $A$& $N$ & $A$\\
\hline
 0 & 0 & 362 & 170 & 827 & 149 & 1251 & 270 & 1656 & 711& 2662 & 1045& 5005 & 994&7546&1573 & 9889 & 1054\\
 10 & 4& 363 & 120 & 828 & 99 & 1252 & 319& 1661 & 1046 & 2761 & 1774& 5104 & 1183& 7557 & 1032& 9988 & 1963\\
 11 & 8 & 382 & 178 & 847 & 157 & 1271 & 278 & 1675 & 670 & 2772 &1053& 5115 & 1002& 7656 & 1671& 9999& 1062 \\
 12 & 3& 383 & 128 & 848 & 107 &1272 & 327 & 1676 & 719 & 2871 & 1872& 5214 & 1281&7766 & 1769\\
\hline
 14 & 2 & 403 & 145 & 867 & 165 & 1291 & 286 & 1695 & 678 & 2882 & 1061& 5225 & 1010& 7777 & 1048\\
16 & 1 & 404 & 95 & 868 & 115 & 1292 & 335 & 1696 & 727 & 2981 & 1970& 5324 & 1379& 7876 & 1867\\
18 & 0 & 423 & 153 & 887 & 173 & 1312 & 352& 1716 & 744 & 2992 & 1069 & 5335 & 1018& 7887 & 1056\\
22 & 7 & 424 & 103 & 888& 123 & 1313 & 401  & 1717 & 793& 3002 & 996& 5434 & 1477& 7986 & 1965\\
\hline
33 & 6 & 443 &161 &908 & 140 &1331 & 1022 & 1736 & 752& 3102 & 1185& 5445 & 1026&7997 & 1064\\
44 & 5 & 444 & 111 &909 & 90  & 1332 & 360& 1737 & 801& 3113 & 1004& 5544 & 1575& 8008 & 991\\
55 & 4 & 463 & 169 &928 & 148  & 1333 & 409  & 1756 & 160& 3212 & 1283& 5555 & 1034& 8107&1180\\
66 & 3 & 464 & 119  & 929 & 98 & 1352 & 368& 1771 & 1054& 3223 & 1012& 5654 & 1673&8118 & 999\\
\hline
77 & 2 & 483 & 177 & 948 & 156 & 1353 & 417 & 1776 & 768& 3322 & 1381& 5665 & 1042& 8217 & 1278 \\
88 & 1 & 484 & 127 & 949 & 106 & 1372 & 376&1777 & 817 & 3333 & 1020& 5764 & 1771& 8228& 1007\\
99 & 0 & 504 & 144 & 968 & 164 & 1373  & 1425& 1796 & 776& 3432 & 1479& 5775 & 1050& 8327 & 1376\\
101 & 98& 505 & 94 &969 & 114 & 1392 & 384&1797 & 825& 3443 & 1028&5874&1869 & 8338 & 1015\\
\hline
110 &17 &524 & 152&988 & 172 & 1393 & 433&1877 & 842& 3542 & 1577&5885 & 1058& 8437 & 1474\\
121&25 & 525 & 102& 989 & 122 & 1413 & 450& 1818 & 891& 3553 & 1036& 5984 & 1967& 8448&1023\\
132&33 & 544 & 160& 1001 & 998 & 1414 & 499& 1837 & 850 & 3652 & 1675& 5995 & 1066& 8547 & 1572\\
141& 114 & 545 & 110&1009 & 148& 1433 & 458&1838 & 899& 1663 & 1044&6006 & 993& 8558 & 1031\\
\hline
143&41 &584 & 176 & 1010 & 107&1434 & 507&1854 & 907& 3762 & 1773& 6105 & 1182& 8657 & 1670\\
154& 49 &585 & 126& 1029 & 156 &1441 & 1030& 1858 & 907& 1773 & 1052& 6215& 1280& 8668 & 1039\\
161&22&605 & 143&1030 & 115 &1453 & 466 &1877 & 866& 3872 & 1871& 6226 & 1009& 8767 & 1768\\
165&57& 606 & 101& 1049 & 164&1454 & 515 &1878 & 915& 3883 & 1060& 6325 & 1378& 8778 & 1047\\
\hline
176& 65&625 & 151&1050 & 123 & 1473 & 474&1881 & 1062 & 3982 & 1969&6336 & 1017&8877 &1866\\
181 & 130 & 626 & 101& 1069 &172&1474 & 523&1897 & 874& 3993 & 1068& 6435 & 1476& 8888 & 1055\\
187 & 73&645 & 159& 1070 & 131&1493 & 482& 1898 & 923& 4004 & 1184& 6446 &1025& 8987 & 1964\\
198 & 81&646 & 109&1089 & 180 &1494 & 531&1918 & 940& 4103 & 1184& 6545 & 1574& 8988& 1063\\
\hline
201 & 147&665 & 167 &1090 & 139&1514 & 548&1938 & 948& 4114 & 1003& 6556 & 1033&9009 & 990\\
202 & 97& 666 & 117&1110 & 156 &1515 & 567&1958 & 956& 4213 & 1282& 6666 & 1041& 9108 & 1179\\
221 & 155&685 & 175&1111 & 205& 1534 & 556&1978 & 964& 4224 & 1011& 6765 & 1770&9119 & 998\\
222 & 105&686 & 125&1130 & 164&1535 & 605& 1991 & 1070& 4323 & 1380& 6875 & 1868& 9218 & 1277\\
\hline
241 & 163&706 &142&1131 & 213&1551 & 1038&1998 & 972& 4334 & 1478& 6886 & 1057&9229& 1006\\
242 & 113&707 & 92& 1150 & 172&1554 & 564&2002 & 997& 4444 & 1027& 6985 & 1966& 9328 & 1375\\
261 & 171&726 & 150&1151 & 221&1555 & 613&2101 & 1186& 4543 & 1576& 6996 & 1065& 9339 & 1014\\
262 & 121&727 & 100&1170 & 180&1574 & 572&2112 & 1005& 4554 & 1035&7007 & 92& 9438 & 1473\\
\hline
281 & 179&746 & 158&1171 & 229 &1575 & 621&2211 & 1284&4654 & 1674& 7106 & 1181& 9449 & 1022\\
282 & 129& 747 & 108& 1190 &  188 &1594 & 580&2222 &1013& 4664 & 1043& 7117 & 1000& 9548 & 1571\\
302 & 146&766 & 166&1191 & 237 &1595 & 629&2332 & 1021& 4763 & 1772& 7216 & 1279 & 9559 & 1030\\
303 & 96&767 & 116&1211 & 254&1615 & 646&2431 & 1480& 4774 & 1051& 7227 & 1008 & 9658 & 669\\
\hline
322 & 154&786 & 174 &1212 & 303&1616 & 695&2442 & 1029& 4873 & 1870& 7326 & 1377 & 9669 & 1038\\
323 & 104&787 & 124&1221 & 1014&1635 & 654&2541 & 578& 4884 & 1059& 7337 & 1016& 9768 & 1767\\
342 & 162& 807 & 141&1231 & 262& 1636 & 703&2552 & 1037& 4983 & 1968& 7436& 1475 & 9779 & 1046\\
343 & 112&808 & 91&1232 & 311 &1655 & 662& 2651 & 1676& 4994 & 1067& 7447 & 1024 & 9878 & 1865\\
\hline
\end{tabular}}
\caption{All 365 wARH numbers less than 10000 and one of their extra term.}
\label{t:1}
\end{table}

We list in Table \ref{t:1} small wARH numbers $N$ and one of their extra terms $A$. We did not find any number that is not an additive extra term. This suggests that the answer to Question \ref{q:6} is negative. We conjecture that all integers are additive extra terms. We observe  from Table \ref{t:1} that certain extra terms, for example $2$, have associated several wARH numbers, respectively $210,55$. The last observation motivates the following definition and questions.

\begin{definition} If $A$ is an additive extra term in a base $b$, let the \emph{multiplicity} of $A$ be the cardinality of the corresponding set of $b$w-ARH numbers.
\end{definition}

\begin{question}\label{q:9} If we fix the multiplicity and the base, is the set of additive extra terms infinite?
\end{question}

\begin{question}\label{q:10-bis} If we fix the base, is the multiplicity of additive rxtra terms bounded?
\end{question}

\section{Examples of wMRH numbers}\label{sec:8-bis}

\begin{table}[h]
\centering
\scalebox{0.8}{
\begin{tabular}{ |c | c | c|c|c|c|c|c|c | c |}
\hline
 $N$ & $A$ & $N$ & $A$ & $N$ & $A$ & $N$ & $A$ & $N$ & $A$\\
\hline
 0 & 0 & 574 & 25 & 1612 & 16, 52 & 3600 & 591 & 5929 & 52\\
 1 & 0&640 & 70 & 1729 & 0, 63 & 3627 & 21, 75& 6400 & 790 \\
 10 & 9 & 736 & 7, 16 & 1855 & 16, 34 & 3640 & 43, 52 & 6624 & 51, 78 \\
 40 & 16& 765 & 33 & 1936 & 25 &4000 & 1996 & 6786 & 51, 60 \\
\hline
 81 & 0 & 810 & 81 & 1944 & 9, 54 &4030 & 123, 303 & 7360 & 214, 304 \\
90 & 21 & 900 & 291 & 2268 & 18, 45 & 4032 & 39, 75 & 7650 & 132, 192\\
100 & 99 & 976 & 39 &2296 & 9, 63 & 4275 & 39, 57& 7663 & 57, 75\\
121 & 7 & 1000 & 999 & 2430& 36, 45 & 4356 & 48  & 7744 & 66\\
\hline
160 & 33 & 1008 &15, 33 &2500 & 493 &4606 & 23, 78 & 8100 & 891\\
250 & 43 & 1089 & 15 &2520 & 11, 201  & 4840 & 204& 8722 & 70, 79\\
252 & 3, 12 & 1207 & 7, 61 &2668 & 7, 70  & 4900 & 687  & 9000 & 2991\\
360 & 51 & 1210 & 106  & 2701 & 27, 63 & 4930 & 42, 69& 9760 & 138, 588\\
\hline
400 & 196 & 1300 & 21, 48 & 2944 & 27, 45 & 5092 & 51, 160 & 9801 & 81\\
403 & 6, 24 & 1458 & 0, 63 & 3025 & 45 & 5605 & 43, 79 \\
484 & 6& 1462 & 21, 30 & 3154 & 25, 70 & 5740 & 124, 94\\
490 & 57& 1600 & 393 &3478 & 25, 52 & 5848 & 43, 61\\
\hline
\end{tabular}}
\caption{All 77 wMRH numbers less than 10000 with all their multiplicative extra terms.}
\label{t:2}
\end{table}

We list in Table \ref{t:2} small wMRH numbers $N$ and all their extra terms $A$. Theorem \ref{thm:3} shows that a wMRH number with multiplier $2$ has at most $7$ digits. A computer search through all integers with at most 7 digits shows that $2$ is not a multiplicative extra term. This motivates Question \ref{q:7}.

We observe from Table \ref{t:2} that certain wMRH numbers, for example, $252$, $403$, and $736$, have several extra terms (respectively $\{3, 12\}$, $\{6, 24\}$, $\{7, 16\})$. This suggests a positive answer to Question \ref{q:22-new}. The table does not show any example of wMRH number with 3 multiplicative extra terms. The smallest example we found is 63504 with extra terms 234, 423, 126.

We also observe  from Table \ref{t:1} that certain extra terms, for example $7$, have associated several wMRH numbers, respectively $121, 736, 1207, 2668$. The last observation motivates the following definition and questions.

\begin{definition} If $A$ is a multiplicative extra term in base $b$, let the \emph{multiplicity} of $A$ be the cardinality of the corresponding set of  $b$-wMRH numbers.
\end{definition}

\begin{question}\label{q:12} If we fix the multiplicity and the base, is the set of multiplicative extra terms infinite?
\end{question}

\begin{question}\label{q:13-bis} If we fix the base, is the multiplicity of multiplicative extra terms bounded?
\end{question}

\section{Conclusion}\label{sec:9}

 In this paper we introduce two new classes of integers. The first class consists of all numbers $N$ for which there exists at least one nonnegative integer $A$, such that the sum of $A$ and the sum of digits of $N$,  added to the reversal of the sum, gives $N$. The second class consists of all numbers $N$ for which there exists at least one nonnegative integer $A$, such that the sum of $A$ and the sum of the digits of $N$, multiplied by the reversal of the sum, gives $N$. All palindromes that either have an even number of digits or an odd number of digits and the middle digit even belong to the first class, and all squares of palindromes with at least two digits belong to the second class. These classes contain and are strictly larger than the classes of $b$-ARH numbers, respectively $b$-MRH numbers introduced in Ni\c tic\u a \cite{N1}. We show many examples of such numbers and ask several questions that may lead to future research. In particular, we try to clarify some of the relationships between these classes of numbers and the well studied class of $b$-Niven numbers. Most of our results are true in a general numeration base.

\bigskip
\hrule
\bigskip

\noindent 2010 {\it Mathematics Subject Classification}:
Primary 11B83; Secondary 11B99.

\noindent \emph{Keywords:} base, $b$-Niven number, reversal, additive $b$-Ramanujan-Hardy number, multiplicative $b$-Ramanujan-Hardy number, high degree $b$-Niven number, palindrome.

\bigskip
\hrule
\bigskip

\bigskip
\hrule
\bigskip

\end{document}